\documentclass[12pt]{amsart}

\usepackage{mathtools}
\mathtoolsset{showonlyrefs=true}
\usepackage[hmargin=0.8in,height=8.6in]{geometry}
\usepackage{amssymb,amsthm, times}
\usepackage{delarray,verbatim}
\usepackage{ifpdf}
\ifpdf
\usepackage[pdftex]{graphicx}
 \else
\usepackage[dvips]{graphicx}
 \fi

\usepackage{bm}

\usepackage{ifpdf}
\usepackage{color}
\definecolor{webgreen}{rgb}{0,.5,0}
\definecolor{webbrown}{rgb}{.8,0,0}
\definecolor{emphcolor}{rgb}{0.5,0.95,0.95}

\usepackage{hyperref}
\hypersetup{%
	colorlinks=true,
	linkcolor=webbrown,
	filecolor=webbrown,
	citecolor=webgreen,
	breaklinks=true}
\ifpdf \hypersetup{pdftex,
	pdfstartview=FitH, 
	bookmarksopen=true,
	bookmarksnumbered=true
} \else \hypersetup{dvips} \fi
\allowdisplaybreaks

\numberwithin{equation}{section}

\newtheorem{theorem}{Theorem}[section]
\newtheorem{proposition}{Proposition}[section]

\newtheorem{remark}{Remark}[section]
\newtheorem{lemma}{Lemma}[section]

\newtheorem{assump}{Assumption}[section]

\numberwithin{remark}{section} \numberwithin{proposition}{section}
\numberwithin{corollary}{section}
\newcommand {\R}{\mathbb{R}}

\newcommand {\cF}{\mathcal{F}}
\newcommand {\A}{\mathcal{A}}

\newcommand {\N}{\mathbb{N}}
\newcommand {\p}{\mathbb{P}}

\newcommand {\E}{\mathbb{E}}

\newcommand{\diff}{{\rm d}}

\newcommand{\lev}{L\'{e}vy }

\newcommand{\Ind}{\mathbb{I}}

\newcommand{\cL}{\mathcal{L}}

\begin{document}
	\title{{On optimal periodic dividend and capital injection strategies for general L\'evy models}}

	\thanks{This version: \today.   }
	\author[D. Mata]{Dante Mata$^{\star}$}
	\thanks{$\star$\, D\'epartement de Math\'ematiques, Universit\'e du Qu\'ebec \`a Montr\'eal (UQAM), 201 av. Pr\'esident-Kennedy, Montr\'eal, QC H2X 3Y7, Canada. Email: mata\_lopez.dante@uqam.ca}
	\thanks{$\sharp$\, Department of Probability and Statistics, Centro de Investigaci\'on en Matem\'aticasA.C. Calle Jalisco
		s/n. C.P. 36240, Guanajuato, Mexico. Email: jluis.garmendia@cimat.mx}
	\author[K. Noba]{Kei Noba$^\ddagger$}
	\thanks{$\ddagger$\, Department of Mathematics, Graduate School of Science, The University of Osaka. 1-1 Machikaneyama, Toyonaka 560-0043, Osaka, Japan. Email: knoba@math.sci.osaka-u.ac.jp}
	\author[J. L. P\'erez]{Jos\'e-Luis P\'erez$^{\sharp}$}
	\date{}
	
	\begin{abstract} 
		We consider a version of de Finetti's dividend problem, with the bail-out contraint to keep the surplus non-negative, and where dividend payments can only be made at the arrival times of an independent Poisson process. For a general \lev process with positive and negative jumps, we show the optimality of a periodic-classical reflection strategy that pays the excess above a given level at each Poisson arrival time, and also reflects below at 0 in the classical sense.\\
		\noindent \small{\noindent  AMS 2020 Subject Classifications: 60G51, 93E20, 91G80\\
			\textbf{Keywords:} Periodic and singular control strategies, L\'evy processes, optimal dividends, capital injection.}
	\end{abstract}
	
	\maketitle
	
	\section{Introduction}
In this paper, we consider the bail-out version of de Finetti's dividend problem, where the goal is to find a joint optimal dividend and capital injection strategy that maximises the expected net present value (NPV) of dividend payments minus the cost of capital injections. Specifically, we {study} the case where a \lev process with both positive and negative jumps models the surplus of an insurance company.

On one hand, we will study an extension of the classical dividend problem by imposing the constraint that dividend payments can only be performed at discrete observation times, independent of the surplus process. This is commonly known as the \textit{periodic} dividend constraint. This constraint aims to reflect the {practical setting in which} the shareholders of a company can only make the decision of paying dividends at discrete times. Early {studies on} periodic decision times {include} Dupuis and Wang \cite{DW2002} and Wang \cite{Wang2001} for the Brownian model, in the context of optimal stopping problems. Regarding the application of the periodic constraint in the dividend problem, {early contributions} include Albrecher et al. \cite{ACT2011, ACT2013} for the compound Poisson model, as well as Albrecher et al. \cite{AGS2011} for the Brownian model. It has been shown that when the underlying dynamics {follow} a \lev process with one-sided jumps, {a barrier strategy—paying out any excess above a certain level at each payment opportunity—is optimal}. {Optimality} of the barrier strategy under the periodic dividend constraint has been shown, for instance, in Cheung and Zhang \cite{CZ2018a} for the compound Poisson model, in Avanzi et al. \cite{ACWW2013} under a spectrally one-sided model with an additional solvency constraint, and in Noba et al. \cite{NobPerYamYanb} under the spectrally negative \lev model. {Our main goal is to extend these results to the more general case of a L\'evy process with two-sided jumps. In particular, we study a version of the periodic dividend constraint in which dividend decision times correspond to the arrival times of a Poisson process, independent of the surplus process.}

On the other hand, the bail-out extension of de Finetti's problem imposes the constraint that capital injections have to be performed to keep the surplus non-negative uniformly in time. In our setting, we assume that {capital injections} can be {made} continuously in time. The bail-out extension has been extensively studied under the spectrally one-sided {L\'evy} model and with various dividend constraints, for instance, in Avram et al. \cite{AvrPalPis2007} in the case of classical dividend controls, in P\'erez et al. \cite{PYY2018} {for} absolutely continuous dividend controls; {and} in Noba et al. \cite{NobPerYu} under the Markov additive model with one-sided jumps. The combination of periodic dividend controls {(in particular, Poisson decision times)} with capital injections has also been studied {for} spectrally one-sided \lev processes. {The Poissonian decision time framework allows the derivation of semi-explicit expressions for the NPV in terms of the so-called scale function.} For instance, Noba et al. \cite{NobPerYamYan} {address} the spectrally negative case; P\'erez and Yamazaki \cite{PerYam} {study} the dual case {(the bail-out dividend problem driven by a spectrally positive L\'evy process)}; and more recently, Moreno-Franco and P\'erez \cite{MFP24} {analyze the additional constraint of transaction costs upon each dividend payment}. {These studies have} shown that a double barrier strategy is optimal, with a suitable upper barrier for dividend payments {and capital injections applied at level $0$ to prevent the surplus from becoming negative. Mata et al. \cite{MMNP} also studied a similar problem—the generalization of Noba et al. \cite{NobPerYamYan}—for spectrally negative Markov additive processes. Notably, in these studies, the scale function plays a crucial role, enabling the authors to propose a candidate optimal dividend barrier and to prove its existence using fluctuation identities for spectrally one-sided L\'evy processes under periodic observations.}

{Extending to processes with two-sided jumps is more challenging, as the scale function is no longer available. To overcome this challenge, we rely on the pathwise properties of the L\'evy process. Our analysis compares the behavior of the sample paths of the controlled L\'evy process when a small perturbation is added to the initial condition.} These comparison results allow us to compute the first derivative of the NPV of a double barrier strategy and, moreover, {to obtain} a concise \textit{stochastic representation} of its first derivative. {With this representation in hand}, we propose a candidate for the optimal strategy, {choosing} the barrier that satisfies a \textit{gradient} condition{—specifically, the conjecture that the slope of the value function equals $1$ at the barrier. We then verify optimality by showing that the {NPV of the} candidate {optimal strategy} satisfies the associated variational inequalities. Noba \cite{Nob2019} applied this pathwise analysis to solve the bail-out dividend problem with classical dividend payments; similarly, in Noba \cite{Nob2023}, the technique was extended to the case of absolutely continuous dividend constraints.}

{To the best of our knowledge, a stochastic control problem for a general L\'evy process under periodic controls was studied in Noba and Yamazaki \cite{NobYam2022}. In their work, to offset the periodic controls, a running cost function is introduced. The main difference in our paper lies in the introduction of an additional classical control representing capital injections, which offsets the dividend payment process and adds complexity to the pathwise analysis, as we must consider the interaction between periodic and classical controls.}

This paper is organised as follows. In Section \ref{Sec:Prelim} we introduce the standing assumptions regarding the underlying \lev process, formulate the bail-out dividend problem under the periodic dividend constraint, and {define} the periodic-classical barrier strategies along {with a proof of their admissibility}. In Section \ref{Sec3}, we perform a pathwise comparison of the controlled processes {after introducing a} perturbation to the initial condition. In Section \ref{sec:candidate}, we prove the existence of our candidate optimal dividend barrier, and in Section \ref{sec:verification}, {we rigorously verify its optimality}. {Proofs of some auxiliary results are presented in the Appendix.}  
\section{Preliminaries}\label{Sec:Prelim}
We introduce and study a version of the bailout dividend problem driven by a general \lev process, where the dividend decisions can only be made at the arrival times of an independent Poisson process.
\subsection{L\'evy processes}
{Let $(\Omega , \mathcal{F}, \p)$  be} a probability space hosting the \lev process $X={( {X(t)}:t \geq 0 )}$.  
For $x\in\R$, we denote by $\p_x$ the law of $X$ when it starts at $x$, and by $\E_x$ its associated expectation operator. Throughout the paper, let $\psi$ denote the characteristic exponent of $X$, which satisfies 
\begin{align}
	\E_0\left[e^{i\lambda {X(t)}}\right]=e^{-t\psi(\lambda)},  \quad \lambda \in \R, ~t\geq 0. 
\end{align}
In particular, we {use the notation} $\mathbb{P} = \mathbb{P}_0$. 
The characteristic exponent $\psi$ is known to take the form  
\begin{align}
	\psi (\lambda) := -i\gamma\lambda +\frac{1}{2}\sigma^2 \lambda^2 
	+\int_{\R \backslash \{0\} } (1-e^{i\lambda x}+i\lambda x1_{\{|x|<1\}}) {\nu}(\diff x) , ~~~~~~\lambda\in\R. \label{202a}
\end{align}
Here, $\gamma\in\R$, $\sigma\geq 0$, and $\nu $ is a measure on $\R \backslash \{0\}$, known as the L\'evy measure, which satisfies 
\begin{align}
	\int_{\R\backslash \{ 0\}}(1\land x^2)  {\nu} (\diff x) < \infty. 
\end{align}
Recall that the process $X$ has paths of bounded variation if and only if $\sigma= 0$ and $\int_{|x|<1} |x|\nu(\diff x)<\infty$. When this holds, we can write
\begin{align}
	\psi(\lambda) := -i\delta{\lambda}+\int_{\R \backslash \{0\} }  (1-e^{i\lambda x}) {\nu} (\diff x),  \label{204}
\end{align}
where
\begin{align}
	\delta := \gamma-\int_{|x|<1}x  {\nu}(\diff x).
\end{align}
Throughout this work we make the following assumptions, {which are necessary to guarantee the well-posedness of the optimisation problem}. 
\begin{assump}\label{assump1}
	We assume that $X$ satisfies $\E[ |X(1)| ] < \infty$.
\end{assump}
\begin{assump}\label{assump1a} We assume that $X$ is not a driftless compound Poisson process.
\end{assump}
We remark that Assumption \ref{assump1} guarantees that the integrability condition \eqref{r_adm} {(introduced in the following section)} is satisfied for the classical reflection control.{ On the other hand, Assumption \ref{assump1a} {implies} that the potential measure of $X$ has no atoms, and it is also used in the proof of Lemma \ref{Lem_Converge} (see e.g. the proof of Lemma 3.2 in \cite{NobYam2022}). 
}
	\subsection{The optimal Poissonian dividend problem with classical capital injection}
\label{Sec202}
	A dividend/capital injection strategy is a pair of processes $\pi :=  (L_r^\pi(t), R_r^\pi(t): t\geq 0)$, where $L_r^\pi$ represents the cumulative amount of dividends, and $R_r^\pi$ is the cumulative amount of capital injections.\\
	Regarding the dividend strategies, we assume that dividend payments can only be made at the arrival times $\mathcal{T}_r := ({T_i}:i\geq 1)$ of a Poisson process $N_r = (N_r(t):t\geq0)$ with intensity $r>0$ which is independent of $X$. In other works, we assume that $L_r^\pi$ admits the form
	\begin{align}
	L_r^\pi(t) = \int_{[0,t]} \varrho^\pi (s) \diff N_r(s), \quad t\geq 0, \label{b001}
	\end{align}
	for some c\`{a}gl\`{a}d process $\varrho^\pi$ adapted to the filtration $\mathbb{F}:=(\mathcal{F}_{t} : t \geq 0)$ 
	which is generated by the processes $X$ and $N_r$.\\
	Regarding the capital injections, we assume that $R_r^\pi$ is a non-decreasing, right-continuous, and $\mathbb{F}-$adapted process, with $R_r^\pi(0-)=0$. Contrary to the dividend payments, capital injections can be made continuously in time. In addition, the process $R_r^\pi$ must satisfy
	\begin{equation}\label{r_adm}
	\E_x \left[ \int_{[0,\infty)} e^{-qt} \diff R_r^\pi(t) \right] < \infty, \quad x \geq 0,
	\end{equation}
	where $q>0$ represents the rate of discounting.\\
	The corresponding risk process associated to the strategy $\pi$ is given by $U_r^\pi(0-) = X(0)$ and
	\[
	U_r^\pi(t) := X(t) - L_r^\pi(t) + R_r^\pi(t), \quad t\geq 0.
	\]
	We denote by $\A$ the set of strategies satisfying the constraints mentioned above and ensuring that $U_r^\pi(t) \geq 0$ for all $t\geq 0$ a.s. We call a strategy $\pi$ \textit{admissible} if $\pi \in \A$.\\
	We consider that $\beta > 1$ is the cost per unit of injected capital, then our aim is to maximise the expected net present value (NPV)
	\[
	v_\pi(x) := \E_x\left[ \int_{[0,\infty)} e^{-qt} \diff L_r^\pi(t) - \beta \int_{[0,\infty)} e^{-qt} \diff R_r^\pi(t) \right], \quad x\geq 0,
	\]
	over the set of admissible strategies $\A$. Hence, the problem is to compute and characterise the value function
	\begin{equation}\label{value}
		v(x) := \sup_{\pi \in \A} v_\pi(x), \quad x \geq 0,
	\end{equation}
	and obtain an optimal strategy $\pi^\ast \in \A$ that attains \eqref{value}
	, if it exists. 
	\subsection{Periodic-classical barrier strategies}\label{Sec203}
	Our objective is to show that a \textit{periodic-classical} barrier strategy $\pi^{0,a}$ is optimal for a suitable selection of the reflection barrier $a\geq 0$.\\
	The controlled process associated with these strategies is known as the \textit{\lev process with periodic reflection above and classical reflection below}, and is denoted by $U_r^{0,a}$. We provide a construction for this process, as well as for the cumulative amounts of periodic and classical reflection, denoted by $L_r^{0,a}$ and $R_r^{0,a}$, respectively. The construction provided here follows the spirit of the one provided in Section 4 of \cite{AvrPalPis2007}.\\
	Set ${\eta}=0$, then the periodic-classic barrier strategy can be defined by repeatedly performing the following operations. Note that $L_r^{0,a}(0)=R_r^{0,a}(0-)=0$. 
	\begin{enumerate}\addtocounter{enumi}{-1}
		\item[($\ast$)] Set  
		\begin{align}
		Z(t) =
		\begin{cases}
		X(t),\qquad &\text{if }{\eta}=0,\\
		a+ X(t) - X({\eta}),\qquad &\text{if }{\eta}>0,
		\end{cases}
		\end{align}
		and let ${\eta}^{\prime} := \inf \lbrace T_i \in \mathcal{T}_r : T_i \geq {\eta} \text{ and } \tilde{U}_r^{0,a}(T_i) > a \rbrace$, where
		\[
		\tilde{R}_r^{0,a}(t) = R_r^{0,a}({\eta}-)  - \inf_{{\eta} \leq s \leq t}\left( Z(t) \wedge 0 \right), \quad \tilde{U}_r^{0,a}(t) = Z(t) + R_r^{0,a}(t) - R_r^{0,a}({\eta}-), \quad t \geq {\eta}.
		\]
		We set $L_r^{0,a}(t) = L_r^{0,a}({\eta})$, $R_r^{0,a}(t) = \tilde{R}_r^{0,a}(t)$ and $U_r^{0,a}(t) = \tilde{U}_r^{0,a}(t)$ for $t \in [{\eta}, {\eta}^{\prime})$, and $L_r^{0,a}({\eta}^{\prime}) = L_r^{0,a}({\eta})+\tilde{U}_r^{0,a}({\eta}^{\prime})-a$, $R_r^{0,a}({\eta}^{\prime}) = \tilde{R}_r^{0,a}({\eta}^{\prime})$ and $U_r^{0,a}({\eta}^{\prime}) = a$. Reset ${\eta} = {\eta}^{\prime}$ and go to the top of Step ($\ast$).
	\end{enumerate}
	It is clear that it admits a decomposition
	\begin{align}
		U_r^{0,a}(t) = X(t) - L_r^{0,a}(t) + R_r^{0,a}(t), \quad t \geq 0, \label{a001}
	\end{align}
	where $L_r^{0,a}(t)$ and $R_r^{0,a}(t)$ are, respectively, the cumulative amounts of periodic and classical reflection until time $t$. {Note} that the resulting process $U^{0, a}_r$ is a strong Markov process. We define the crossing times
	\begin{equation}\label{cross_times}\tau_0^- := \inf\lbrace t \geq 0 : X(t) < 0 \rbrace \text{ and } T_a^+ := \inf\lbrace T_i \in \mathcal{T}_r : X(T_i) > a \rbrace. \end{equation}
\begin{remark} \label{Rem201}
The periodic-classical reflection strategy is indeed admissible. 
In order to show this, let ${T_0}=0$, and recall that  for each $i\geq 1$, ${T_i}$ denotes the $i$-th jump time of the Poisson process with rate $r>0$. Then, by the strong Markov property we have, for $x\in\R$, 
\begin{align}
\E_x \left[ \int_{[0,\infty)} e^{-qt} \diff R_r^{0,a}(t) \right] &= \sum_{i=1}^{\infty}  \E_x \left[ \int_{[{T_{i-1}},{T_i})} e^{-qt} \diff R_r^{0,a}(t) \right] \notag \\
&= \sum_{i=1}^{\infty} \E_x\left[ e^{-q T_{i-1}}  \E_{U_r^{0,a}({T_{i-1}})} \left[ \int_{[0,T_1)} e^{-qt} \diff R_r^{0,a}(t) \right] \right]-(0\land x)\notag \\
&\leq \sum_{i=1}^{\infty} \E_x\left[ e^{-q T_{i-1}}\right]  \E_{0} \left[ \int_{[0,T_1)} e^{-qt} \diff  (-\underline{X}(t)) \right]-(0\land x) \label{adm2} \\
&\leq \sum_{i=1}^{\infty} (M_e)^{i-1} \E_{0} \left[ \int_{[0,\infty)} e^{-qt} \diff  (-\underline{X}(t)) \right] -(0\land x)\\
& =\left(\sum_{i=1}^{\infty} (M_e)^{i-1} \right) \left( 
\sum_{k=0}^\infty\E_{0} \left[ \int_{[k,k+1)} e^{-qt} \diff  (-\underline{X}(t)) \right]
\right)-(0\land x)\\
&\leq \left(\sum_{i=1}^{\infty} (M_e)^{i-1} \right) \left( 
\sum_{k=0}^\infty e^{-q k}\E_0 \left[ -\left(\inf_{t \in [0, 1]} X(t)  \right)\right]
\right)-(0\land x),
 \label{adm1}
\end{align}
where $M_e := \E\left[ e^{-q \zeta} \right]$ and $\zeta$ is an exponential random variable with parameter $r$. 
Finally, thanks to Assumption \ref{assump1} and (3.10) of \cite{Nob2019}, we obtain that the right-hand side of \eqref{adm1} is finite.\\
To conclude, we note that we can write $L^{0,a}_r$ in terms of the $\mathbb{F}$-c\`agl\`ad process $\varrho^{0,a} = ( \varrho^{0,a}(t) )_{t\geq 0}$ with
\[
\varrho^{0,a}(t) = (U^{0,a}_r(t-)-a)^+, \quad t\geq 0.
\]
It now follows that we can write
\[
L^{0,a}_r(t) = \int_{[0,t]} \varrho^{0,a}(s) \diff N_r(s), \quad t\geq 0.
\]

\end{remark}

We also introduce the \lev process with periodic reflection above at a level $a\geq 0$, denoted by $U_r^{a}$, as it will be used in various proofs and constructions in this paper. This process is constructed inductively as follows. 
\\
{Set ${\eta}= 0$. 
\begin{enumerate}\addtocounter{enumi}{-1}
	\item[($\ast$)] Set  
	\begin{align}
		Z(t) =
		\begin{cases}
			X(t),\qquad &\text{if }\eta=0,\\
			a+ X(t) - X({\eta}),\qquad &\text{if }{\eta}>0,
		\end{cases}
	\end{align}
	and let ${\eta}^{\prime} := \inf \lbrace T_i \in \mathcal{T}_r : T_i \geq {\eta} \text{ and } \tilde{U}_r^{0,a}(T_i) > a \rbrace$, where
	\[
	\tilde{U}_r^{0,a}(t) = Z(t), \quad t \geq {\eta}.
	\]
	We set $U_r^{0,a}(t) = \tilde{U}_r^{0,a}(t)$ for $t \in [{\eta}, {\eta}^{\prime})$, and $U_r^{0,a}({\eta}^{\prime}) = a$. Reset ${\eta} = {\eta}^{\prime}$ and go to the top of Step ($\ast$).
\end{enumerate}
}
Note that the resulting process $U_r^a$ is a strong Markov process. We define the first down-crossing time below 0 of the process $U_r^a$ as
\begin{equation}\label{fpt_poisson}
	\kappa_0^{a,-}(r) := \inf \left\lbrace t\geq 0: U_r^{a}(t) < 0 \right\rbrace.
\end{equation}
{Before going into the main result and its proof, we make one more assumption on $X$. We remark that this is a technical assumption that will be used in the verification of optimality, however the results from Sections \ref{Sec3} and \ref{sec:candidate} remain valid independently of this condition.}
\begin{assump}\label{assump2a} When the process $X$ has paths of unbounded variation, we assume, for any $a >0$, that the mapping $x \mapsto \E_{x}\left[ e^{-q \kappa_0^{a,-}(r)} ; \kappa_0^{a,-}(r)<\infty \right]$
has an a.e.-continuous Radon-Nykodym density  bounded on every compact set of $(0,\infty)$. 
\end{assump}
{The following lemma gives an example of processes which {satisfy} the Assumption \ref{assump2a}. We put its proof in Appendix \ref{App002}.}
\begin{lemma}\label{lem_smooth}
	If {$X$ has unbounded variation paths and} the \lev measure satisfies $\nu(0,\infty) < \infty$ or $\nu(-\infty,0) < \infty$, then Assumption \ref{assump2a} is satisfied.
\end{lemma}

	\section{Behaviour of $U_a^{(x+\varepsilon)}$ and $U_a^{(x)}$ under $\p_0$}\label{Sec3}
	We analyse the behaviour of the controlled process under a periodic-classical barrier strategy at $a \geq 0$ and 0, respectively, when the initial condition is perturbed. Throughout this section, we consider that $X$ is a \lev process started at 0, and for $y\in \R$ we write $X^{(y)}(t) := X(t) + y$.\\
	In addition, for the remainder of this section we fix a periodic-classical barrier strategy $\pi^{0,a}$, and we write its associated controlled process as $U^{(y)}$. We remark that the superscript denotes that it is driven by the \lev process $X^{(y)}$. {Similarly, we denote the cumulative amount of dividend payments by $L^{(y)}$, and the cumulative amount of capital injections by $R^{(y)}$.
	
	Additionally, we write $T_a^{(y),+} := \inf \lbrace T_k \in \mathcal{T}_r : X^{(y)} (T_k) > a \rbrace$ and $\tau_0^{(y),-} := \inf\lbrace t\geq 0 : X^{(y)}(t) < 0 \rbrace$.\\
	We now state our main results in this section. 
	\begin{theorem}\label{Prop_diff}
	For $x \geq 0$ and $\varepsilon > 0$ we have:
	\begin{enumerate}
		\item[(i)] $R^{(x+\varepsilon)}(t) - R^{(x)}(t)$ is non-increasing and takes values in $[-\varepsilon, 0]$.
		\item[(ii)] $U^{(x+\varepsilon)}(t) - U^{(x)}(t)$ is non-increasing and takes values in $[0,\varepsilon]$.
		\item[(iii)] $L^{(x+\varepsilon)}(t) - L^{(x)}(t)$ is non-decreasing and takes values in $[0,\varepsilon]$.
	\end{enumerate}
	\end{theorem}
	\begin{proof}
		{Fix $\varepsilon>0$.} We define
		\[
		\zeta(k) := U^{(x+\varepsilon)}(T_k) - U^{(x)}(T_k),\qquad k \in \N \cup \{0\}.
		\]
		We will prove items (i)--(iii) by induction. 
		First, we note that $T_0 = 0$, also we note that at time 0, items (i)--(iii) hold trivially since $U^{(x+\varepsilon)}(0) - U^{(x)}(0) = \varepsilon$, and by definition we have $R^{(x+\varepsilon)}(0) - R^{(x)}(0) = 0$ and $L^{(x+\varepsilon)}(0) - L^{(x)}(0) = 0$. 
		Fix $k \in \N \cup \{0\}$, and assume that (i)--(iii) hold true on the time interval $[0,T_k]$. We will then prove that the claim is true on the interval $(T_k, T_{k+1}]$.
		\\
		We start with the proof of (i). By the definition of the periodic-classical barrier strategy, on the interval $(T_k, T_{k+1})$, the processes $U^{(x+\varepsilon)}$ and $R^{(x+\varepsilon)}$ (resp. $U^{(x)}$ and $R^{(x)}$) correspond to the \lev process with classical reflection below at $0$ and the cumulative amount of capital injections, driven by $X^{(U^{(x+\varepsilon)}(T_{k}))}$ (resp. $X^{(U^{(x)}(T_{k}))}$).
		
		Due to this observation, it is a consequence of the proof of Lemma 7 in \cite{NobYam2020} that $t \mapsto R^{(x+\varepsilon)}(t) - R^{(x)}(t)$ is non-increasing on $[T_k, T_{k+1})$ and takes values in $[-\zeta(k),0] \subseteq [-\varepsilon,0]$. To conclude the argument, since the underlying \lev process is independent of the Poisson arrival times, then the probability that $X$ has a negative jump exactly at time $T_{k+1}$ is equal to 0. This implies that $\Delta R^{(x+\varepsilon)}(T_{k+1}) = 0$ and $\Delta R^{(x)}(T_{k+1}) = 0$ with probability 1, so that $$R^{(x+\varepsilon)}(T_{k+1}) - R^{(x)}(T_{k+1})=R^{(x+\varepsilon)}(T_{k+1}-) - R^{(x)}(T_{k+1}-),$$ thus (i) holds true on $[T_k, T_{k+1}]$.
		
		For the proof of item (ii), we recall that on $[T_k,T_{k+1})$ only capital injections are performed, while dividend payments take place only at time $T_{k+1}$. It is a consequence of the proof of Lemma 7 in \cite{NobYam2020}, and item (i) that: 
		\begin{equation}\label{301}
			U^{(x+\varepsilon)}(t) - U^{(x)}(t) \in [0, \zeta(k)] \subseteq [0,\varepsilon] \text{ and is non-increasing for } t \in [T_k, T_{k+1}).
		\end{equation}
		To conclude the analysis of item (ii), we study the behaviour of $ U^{(x+\varepsilon)} - U^{(x)} $ at time $T_{k+1}$ regarding the possible instances of dividend payments. We define the sets:
		\begin{equation}\label{A_delta_set}
			A^{(\delta)} := \lbrace j\geq 1 : \Delta L^{(x+\delta)}(T_{j}) > 0 \rbrace, \quad \delta \in \lbrace 0,\varepsilon \rbrace.
		\end{equation}
		Since $U^{(x+\varepsilon)}(T_{k+1}-) \geq  U^{(x)}(T_{k+1}- ) $, we have the following relation:
		\begin{equation}\label{set1}
			\lbrace k+1 \in A^{(0)} \rbrace \subset \lbrace k+1\in A^{(\varepsilon)} \rbrace.
		\end{equation}
		Based on this relation, we distinguish the following cases regarding dividend payments.
		\begin{itemize}
			\item[(A)] If $k+1 \in A^{(0)}$, then due to \eqref{set1} both $U^{(x)} \text{ and } U^{(x+\varepsilon)}$ have dividend payments at time $T_{k+1}$. As a consequence, $U^{(x)}(T_{k+1} ) = a$ and $U^{(x+\varepsilon)}(T_{k+1} ) = a$, hence $$\zeta(k+1)=0. $$
			\item[(B)] If $k+1 \notin A^{(0)}$, then we have $U^{(x)}(T_{k+1}-)\leq a$, hence $\Delta U^{(x)}(T_{k+1}) = 0$. In addition, we have:
			\begin{enumerate}
				\item If $k+1\notin A^{(\varepsilon)}$, then $ U^{(x+\varepsilon)}(T_{k+1}-)\leq a $ as well, hence $\Delta U^{(x+\varepsilon)}(T_{k+1}) = 0$. It follows from \eqref{301} that
				\begin{align*}
				\zeta(k+1) = U^{(x+\varepsilon)}(T_{k+1}-) - U^{(x)}(T_{k+1}-) \in [0,\zeta(k)].
				\end{align*}
				\item If $k+1\in A^{(\varepsilon)}$, then $ U^{(x+\varepsilon)}(T_{k+1}-) > a $. Consequently, we have $\Delta U^{(x+\varepsilon)}(T_{k+1}) = a - U^{(x+\varepsilon)}(T_{k+1}-)$, and also 
				\begin{align*}
					0\leq\zeta(k+1) &= a - U^{(x)}(T_{k+1}-) \\
					& < U^{(x+\varepsilon)}(T_{k+1}-) - U^{(x)}(T_{k+1}-) \\
					& \leq \zeta(k),
				\end{align*}
				where in the second inequality we have used \eqref{301}.
			\end{enumerate}
		\end{itemize}
		Summing up, we conclude that claim (ii) holds true on $[T_k, T_{k+1}]$.
		
		Now, we prove item (iii). First, due to the definition of the periodic-classical barrier strategy, there are no dividend payment opportunities on $(T_{k}, T_{k+1})$, hence $t \mapsto L^{(x+\varepsilon)}(t) - L^{(x)}(t)$ is constant on $(T_{k}, T_{k+1})$. We now study the behaviour of $L^{(x+\varepsilon)} - L^{(x)}$ at time $T_{k+1}$. Recall the sets $A^{(\delta)}$ as defined in \eqref{A_delta_set}, as well as the relation \eqref{set1}. We observe the following cases regarding the possible instances of dividend payments upon time $T_{k+1}$.
		\begin{itemize}
			\item[(A)] If $k+1 \in A^{(0)}$, then due to \eqref{set1} both $U^{(x)} \text{ and } U^{(x+\varepsilon)}$ have dividend payments at time $T_{k+1}$. 
			In addition, because there are no capital injections at time $T_{k+1}$, $$\Delta (L^{(x+\varepsilon)}-L^{(x)})(T_{k+1}) = U^{(x+\varepsilon)}(T_{k+1}-) - U^{(x)}(T_{k+1}-) \geq 0,$$ where the inequality follows directly from \eqref{301}. 
			\item[(B)] If $k+1 \notin A^{(0)}$, then we have $U^{(x)}(T_{k+1}-)\leq a$, hence $\Delta L^{(x)}(T_{k+1}) = 0$. In addition, we consider the following subcases:
			\begin{enumerate}
				\item If $k+1\notin A^{(\varepsilon)}$, then $\Delta L^{(x+\varepsilon)}(T_{k+1}) = 0$ as well. Then it follows easily that $$\Delta (L^{(x+\varepsilon)}-L^{(x)})(T_{k+1}) = 0.$$

				\item If $k+1 \in A^{(\varepsilon)}$, then $U^{(x+\varepsilon)}(T_{k+1}-) \geq a$. Consequently, we have $\Delta L^{(x+\varepsilon)}(T_{k+1}) = U^{(x+\varepsilon)}(T_{k+1}-) - a \geq 0$, hence $$\Delta (L^{(x+\varepsilon)}-L^{(x)})(T_{k+1}) \geq 0.$$
			\end{enumerate}
		\end{itemize}
		To sum up, from cases (A) and (B) we conclude that $L^{(x+\varepsilon)} - L^{(x)}$ is non-negative and non-decreasing on $[T_k, T_{k+1}]$.\\
	To conclude our induction argument, it remains to prove that $L^{(x+\varepsilon)} - L^{(x)}$ takes values on $[0,\varepsilon]$. To show this, using \eqref{a001} we obtain for $t \in [T_k, T_{k+1}]$
	\begin{align}
		\left(L^{(x+\varepsilon)}(t) - L^{(x)}(t) \right)-\left(R^{(x+\varepsilon)}(t) - R^{(x)}(t) \right)=&
		\left(X^{(x+\varepsilon)}(t) - X^{(x)}(t) \right)-\left(U^{(x+\varepsilon)}(t) - U^{(x)}(t) \right)
		\\
		=&\varepsilon- \left(U^{(x+\varepsilon)}(t) - U^{(x)}(t) \right)\\ \leq & \varepsilon, \label{302}
	\end{align}
	where the inequality follows from item (ii).\\
	Recall that we have already established that $L^{(x+\varepsilon)} - L^{(x)}$ takes non-negative values. For the upper bound, from item (i) we have that $- \left( R^{(x+\varepsilon)} - R^{(x)} \right)$ takes values on $[0,\varepsilon]$, hence from \eqref{302} we have for $t \in [T_k, T_{k+1}]$
	\begin{align*}
		\left(L^{(x+\varepsilon)}(t) - L^{(x)}(t) \right) &\leq \varepsilon + \left(R^{(x+\varepsilon)}(t) - R^{(x)}(t) \right) \\
		& \leq \varepsilon.
	\end{align*}
	Thus, the induction argument is complete, and we conclude that items (i)--(iii) hold true on $[0,\infty)$.
	\end{proof}
	\begin{lemma}
	For $x \geq 0$ and $\varepsilon > 0$, on the event $\{ \tau_0^{(x),-} < T_a^{(x),+} \}$, we have the following for the capital injections:
	\begin{align}
		\inf\{t>0: R^{(x+\varepsilon)}(t) - R^{(x)}(t) < 0\} &\geq \tau^{(x),-}_0.
		\label{aa02}
	\end{align}
	On the other hand, on the event $\{\tau_0^{(x+\varepsilon),-}<T_a^{(x+\varepsilon),+}\}$, we have:
	\begin{equation}\label{diff_R_2}
		L^{(x+\varepsilon)}(t) - L^{(x)}(t) = 0\quad t\in[0, \infty) ,
		\quad R^{(x+\varepsilon)}(t) - R^{(x)}(t)=\begin{cases}
			0, \quad t \in [0, \tau_0^{(x),-}) \\
			-\varepsilon, \quad t \in [\tau_0^{(x+\varepsilon),-},\infty).
		\end{cases}
	\end{equation}
	Regarding the dividend payments, on the event $\{ T_a^{(x+\varepsilon),+} < \tau_0^{(x+\varepsilon),-} \}$, we have:
			\begin{align}
		\inf\{t>0: L^{(x+\varepsilon)}(t) - L^{(x)}(t) > 0\} &\geq T^{(x+\varepsilon),+}_a.
		\label{aa01}
	\end{align}
	On the other hand, on the event $\{T_a^{(x),+}<\tau_0^{(x),-}\}$, the following holds:
	\begin{equation}\label{diff_L_2}
		L^{(x+\varepsilon)}(t) - L^{(x)}(t) =  \begin{cases}
			0, \quad t \in [0, T_a^{(x+\varepsilon),+}) \\
			\varepsilon, \quad t \in [T_a^{(x),+},\infty)
		\end{cases} ,
		\quad R^{(x+\varepsilon)}(t) - R^{(x)}(t)=0, \quad t\in [0, \infty).
	\end{equation}
	\end{lemma}
	\begin{proof}
		For the proof of \eqref{aa02}, on $\{ \tau_0^{(x),-} < T_a^{(x),+} \}$, we have $U^{(x)}(t) = X^{(x)}(t)$ and $R^{(x)}(t) = 0$ for all $t \in [0,\tau_0^{(x),-})$. Moreover,
	\begin{align}
&\label{eq001} \tau^{(x),-}_0= \inf\{t\geq 0 : R^{(x)}(t)>0\} ,  
\end{align} 
		 by the construction of periodic-classical barrier strategy. Since $R^{(x+\varepsilon)}$ is non-negative and by \eqref{eq001}, we have \eqref{aa02}.
Similarly, for the proof of \eqref{aa01}, on $\{  T_a^{(x+\varepsilon),+}<\tau_0^{(x+\varepsilon),-}\}$, we have $U^{(x+\varepsilon)}(t) = X^{(x+\varepsilon)}(t)$ and $L^{(x+\varepsilon)}(t) = 0$ for all $t \in [0,T_a^{(x+\varepsilon),+})$. Moreover,  
\begin{align}
T_a^{(x+\varepsilon),+}=\inf\{t\geq 0:L^{(x+\varepsilon)}(t)>0\},
\label{eq002}
\end{align}
by the construction of periodic-classical barrier strategy. 
Since $L^{(x)}$ is non-negative and by \eqref{eq002}, we obtain \eqref{aa01}.

We proceed to the proof of \eqref{diff_R_2}. 
The first identity for $R^{(x+\varepsilon)}(t)-R^{(x)}(t)$ in \eqref{diff_R_2} follows from \eqref{aa02} and the fact that $\{\tau_0^{(x+\varepsilon),-}<T_a^{(x+\varepsilon),+}\}\subset\{\tau_0^{(x),-}<T_a^{(x),+}\}${, which can be deduced by applying Theorem \ref{Prop_diff} together with \eqref{eq001} and \eqref{eq002} (for $x$ and $x+\varepsilon$)}. 
By the construction of the periodic-classical barrier strategy, on $\{\tau_0^{(x+\varepsilon),-} < T_a^{(x+\varepsilon),+}\}$, we have
\[ 
U^{(x+\varepsilon)}(t) = X^{(x+\varepsilon)}(t),\, L^{(x+\varepsilon)}(t) = 0 \text{ and } R^{(x+\varepsilon)}(t) = 0, \quad \text{for } t \in [0,\tau_0^{(x+\varepsilon),-}),\]
as well as $U^{(x+\varepsilon)}\left(\tau^{(x+\varepsilon),-}_0\right)=0$. 
By Theorem \ref{Prop_diff}(ii) and (iii), and the non-negativity of $L^{(x)}$ and $ U^{(x)} $, we also have 
\[
L^{(x)}(t)=0\qquad \text{for $t \in [0,\tau_0^{(x+\varepsilon),-})$,} 
\]
and $U^{(x)}\left(\tau_0^{(x+\varepsilon),-}\right)=0$. Therefore,
\[ 
L^{(x+\varepsilon)}(t)-L^{(x)}(t)=0\qquad\text{for $t \in [0,\tau_0^{(x+\varepsilon),-})$}.
\] 
Furthermore, by \eqref{a001}, we have
\begin{align}
	&R^{(x+\varepsilon)}(\tau_0^{(x+\varepsilon),-}) - R^{(x)}(\tau_0^{(x+\varepsilon),-}) \\
	&=\left(U^{(x+\varepsilon)}(\tau_0^{(x+\varepsilon),-}) - U^{(x)}(\tau_0^{(x+\varepsilon),-}) \right)+
		\left(L^{(x+\varepsilon)}(\tau_0^{(x+\varepsilon),-}) - L^{(x)}(\tau_0^{(x+\varepsilon),-}) \right)\\
&\quad		-\left(X^{(x+\varepsilon)}(\tau_0^{(x+\varepsilon),-}) - X^{(x)}(\tau_0^{(x+\varepsilon),-}) \right)
		\\
		&=0+0-\varepsilon=-\varepsilon.
	\end{align}
	Therefore, combining the above with Theorem \ref{Prop_diff}(i) and (ii), we obtain the second identity for $R^{(x+\varepsilon)}(t)-R^{(x)}(t)$ in \eqref{diff_R_2}, as well as $U^{(x+\varepsilon)}(t)-U^{(x)}(t)=0$ for $t\geq \tau_0^{(x+\varepsilon),-}$. 
Then, by \eqref{a001}, we have for $t\geq \tau_0^{(x+\varepsilon),-}$, 
\begin{align}
	L^{(x+\varepsilon)}(t) - L^{(x)}(t) 
	&=\left(X^{(x+\varepsilon)}(t) - X^{(x)}(t) \right)+
		\left(R^{(x+\varepsilon)}(t) - R^{(x)}(t) \right)\\
&\quad		-\left(U^{(x+\varepsilon)}(t) - U^{(x)}(t) \right)
		\\
		&=\varepsilon-\varepsilon-0=0.
	\end{align}
	Thus, the proof of \eqref{diff_R_2} is complete. 
We proceed to the proof of \eqref{diff_L_2}. 
The first identity for $L^{(x+\varepsilon)}(t)-L^{(x)}(t)$ in \eqref{diff_L_2} follows from \eqref{aa01} and the fact that $\{ T_a^{(x),+} < \tau_0^{(x),-} \}\subset\{ T_a^{(x+\varepsilon),+} < \tau_0^{(x+\varepsilon),-} \}${, which can be obtained by applying Theorem \ref{Prop_diff} together with \eqref{eq001} and \eqref{eq002} (for $x$ and $x+\varepsilon$).} 
By the construction of the periodic-classical barrier strategy, we have, on $\{ T_a^{(x),+} < \tau_0^{(x),-} \}$, 
\[
U^{(x)}(t) = X^{(x)}(t), \ L^{(x)}(t) = 0, \ \text{and} \ R^{(x)}(t) = 0, \ \text{for all $t \in [0,T_a^{(x),+})$},
\]
with $U^{(x)}(T_a^{(x),+})=a$. 

By Theorem \ref{Prop_diff}(i), \eqref{set1} and the non-negativity of $R^{(x)}$, we also have 
\[
R^{(x+\varepsilon)}(t)=0,\qquad\text{for $t \in [0,T_a^{(x),+})$,} 
\]
and $U^{(x+\varepsilon)}(T_a^{(x),+})=a$. 
Thus, $R^{(x+\varepsilon)}(t)-R^{(x)}(t)=0$ for $t \in [0,T_a^{(x),+})$. 

Furthermore, by \eqref{a001}, we have
\begin{align}
	L^{(x+\varepsilon)}(T_a^{(x),+}) - L^{(x)}(T_a^{(x),+}) 
	&=\left(X^{(x+\varepsilon)}(T_a^{(x),+}) - X^{(x)}(T_a^{(x),+}) \right)+
		\left(R^{(x+\varepsilon)}(T_a^{(x),+}) - R^{(x)}(T_a^{(x),+}) \right)\\
&\quad		-\left(U^{(x+\varepsilon)}(T_a^{(x),+}) - U^{(x)}(T_a^{(x),+}) \right)
		\\
		&=\varepsilon-0-0=\varepsilon.
	\end{align}
	Therefore, combining the above with Theorem \ref{Prop_diff}(ii) and (iii), we obtain the second identity for $L^{(x+\varepsilon)}(t)-L^{(x)}(t)$ in \eqref{diff_L_2} as well as $U^{(x+\varepsilon)}(t)-U^{(x)}(t)=0$ for $t\geq T_a^{(x),+}$. 
Hence, by \eqref{a001}, we have for $t\geq T_a^{(x),+}$, 
\begin{align}
	R^{(x+\varepsilon)}(t) - R^{(x)}(t) 
	&=\left(U^{(x+\varepsilon)}(t) - U^{(x)}(t) \right)+
		\left(L^{(x+\varepsilon)}(t) - L^{(x)}(t) \right)\\
&\quad		-\left(X^{(x+\varepsilon)}(t) - X^{(x)}(t) \right)
		\\
		&=0+\varepsilon-\varepsilon=0.
	\end{align}
Thus, the proof of \eqref{diff_L_2} is complete. 
	\end{proof}
	\section{Selection of a candidate optimal threshold}\label{sec:candidate}
		 We focus on the periodic-classical barrier strategy and we characterise our candidate optimal barrier, denoted by $a^{*}$, which is chosen to satisfy the  \textit{smooth fit} condition $v_{a^{*}}'(a^{*})=1$. {Throughout this section, we recall the first passage time below 0, denoted by $\tau_0^{-}$, and the Poissonian exit time above $a \in \R$, denoted by $T_a^+$, which were both defined in \eqref{cross_times}.}
	
	{Later in this section we will need the following auxiliary result, whose proof is omitted since it is similar to that} of Lemma 3.2 in \cite{NobYam2022}. 
	\begin{lemma}\label{Lem_Converge}
		{For fixed $a \in \R$, we have $\lim_{a' \rightarrow a} T^{+}_{a'} = T^{+}_{a}$ $\p$-a.s.}
	\end{lemma}
{First, it is clear that we can write the NPV of a periodic-classical barrier strategy as $v_a(x) = v_a^{L}(x) - \beta v_a^{R}(x)$ for all $x \in [0,\infty)$, where
\[
v_a^{L}(x):= \E_x\left[\int_{[0,\infty)} e^{-qt} \diff L_r^{0,a}(t)\right], \quad v_a^{R}(x):= \E_x\left[\int_{[0,\infty)} e^{-qt} \diff R_r^{0,a}(t)\right].
\]
} The next result allows us to express the derivatives of $v_a^{L}$ and $v_a^{R}$ in terms of the crossing times $T_a^{+}$ and $\tau_0^{-}$.
\begin{proposition}\label{Prop401}
	Fix $a \geq 0$. Then the functions $v_a^{L}$ and $v_a^{R}$ are continuously differentiable on $(0,\infty)$, and their derivatives have the form 
	\[
	v_a^{L \prime}(x) = \E_x\left[ e^{-q T_a^{+}} ; T_a^{+} < \tau_0^{-} \right], \quad v_a^{R \prime}(x) = -\E_x \left[ e^{-q \tau_0^{-}} ; \tau_0^{-} < T_a^{+} \right], \quad x \in (0,\infty). 
	\]
\end{proposition}
\begin{proof}
	{Throughout this proof we will use that $\E \left[ e^{- q \tau_0^{(x),-}} ; \tau_0^{(x),-} < T_a^{(x),+}\right]=\E_x \left[ e^{- q \tau_0^{-}} ; \tau_0^{-} < T_a^{+}\right]$, as well as $\E \left[ e^{- q T_a^{(x),+}} ;  T_a^{(x),+} < \tau_0^{(x),-} \right] = \E_x \left[ e^{- q T_a^{+}} ;  T_a^{+} < \tau_0^{-} \right]$.\\	}
	\textit{Step (i).-} We compute the right- and left-derivatives of $v_a^{L}$. {Using} the notation {introduced in} Section \ref{Sec3}, for $\varepsilon > 0$ and $x \in (0,\infty)$ we rewrite $v_a^{L}(x+\varepsilon) - v_a^{L}(x) $ as
	\begin{equation}\label{v_L_diff}
	v_a^{L}(x+\varepsilon) - v_a^{L}(x) = \E \left[ \int_{{[0, \infty)}} e^{-qt} \diff (L^{(x+\varepsilon)}(t) - L^{(x)}(t)) \right].
	\end{equation}
	Hence, we obtain 
	\begin{align}
		v_a^{L}(x+\varepsilon) - v_a^{L}(x) &= \E \left[ \int_{[0,\infty)} q e^{-qt} \left( L^{(x+\varepsilon)}(t) - L^{(x)}(t) \right)  \diff t \right] \\
		& {=} \E \left[  \int_{[T_a^{(x+\varepsilon),+}  ,\infty)} q e^{-qt} \left( L^{(x+\varepsilon)}(t) - L^{(x)}(t) \right)  \diff t ; T_a^{(x+\varepsilon),+} < \tau_0^{(x+\varepsilon),-} \right] \\
		& \leq \varepsilon \E \left[ e^{- q T_a^{(x+\varepsilon),+}} ; T_a^{(x+\varepsilon),+} < \tau_0^{(x+\varepsilon),-} \right] 
		{=\varepsilon \E_x \left[ e^{- q T_{a-\varepsilon}^{+}} ; T_{a-\varepsilon}^{+} < \tau_{-\varepsilon}^{-} \right] }, \label{right1}
	\end{align}
	 where in the first equality we used integration by parts, in the second equality we used \eqref{aa01} and the identity concerning $L^{(x+\varepsilon)}-L^{(x)}$ in \eqref{diff_R_2}, and in the first inequality we used Theorem \ref{Prop_diff}(iii).\\
	On the other hand, we have
	\begin{align}
		v_a^{L}(x+\varepsilon) - v_a^{L}(x) &= \E \left[ \int_{[0,\infty)} q e^{-qt} \left( L^{(x+\varepsilon)}(t) - L^{(x)}(t) \right)  \diff t \right] \\
		&\geq\E \left[  \int_{[T_a^{(x),+},\infty)} q e^{-qt} \left( L^{(x+\varepsilon)}(t) - L^{(x)}(t) \right)  \diff t ; T_a^{(x),+} < \tau_0^{(x),-} \right] \\
		& = \varepsilon \E \left[ e^{- q T_a^{(x),+}} ; T_a^{(x),+} < \tau_0^{(x),-} \right]
		=\varepsilon \E_x \left[ e^{- q T_a^{+}} ; T_a^{+} < \tau_0^{-} \right]
		. \label{right2}
	\end{align}
		where in the first equality we used integration by parts and in the second {equality} we used the identity concerning $L^{(x+\varepsilon)}-L^{(x)}$ in \eqref{diff_L_2}. \\
	By combining \eqref{right1} and \eqref{right2}, we get the following bounds for $x\in(0,\infty)$,
	\begin{equation}\label{right2a}
	\varepsilon \E_x \left[ e^{- q T_a^{+}} ; T_a^{+} < \tau_0^{-} \right]
	\leq v_a^{L}(x+\varepsilon) - v_a^{L}(x) \leq 
\varepsilon \E_x \left[ e^{- q T_{a-\varepsilon}^{+}} ; T_{a-\varepsilon}^{+} < \tau_{-\varepsilon}^{-} \right]	.
	\end{equation}
	Following similar arguments to those used to derive \eqref{right1} and \eqref{right2} give for $x > 0$ and{ $\varepsilon \in (0,x)$: } 
	\begin{equation}\label{left1}
		\varepsilon
		\E_x \left[ e^{- q T_{a+\varepsilon}^{+}} ; T_{a+\varepsilon}^{+} < \tau_{\varepsilon}^{-} \right]
		\leq v_a^{L}(x) - v_a^{L}(x-\varepsilon) \leq \varepsilon 
		\E_x \left[ e^{- q T_{a}^{+}} ; T_{a}^{+} < \tau_{0}^{-} \right].
	\end{equation}
	Thus, by \eqref{right2a} and \eqref{left1}, together with an application of Lemma \ref{Lem_Converge}, and Lemma 1(ii) in \cite{NobYam2020}, we have, for $x\in(0, \infty)$,
	\begin{equation}\label{limits1}
		\lim_{\varepsilon\downarrow 0} \frac{v_a^{L}(x+\varepsilon) - v_a^{L}(x)}{\varepsilon} = \lim_{\varepsilon\downarrow 0} \frac{v_a^{L}(x) - v_a^{L}(x-\varepsilon)}{\varepsilon} = \E_x \left[ e^{- q T_a^{+}} ; T_a^{+} < \tau_0^{-} \right].
	\end{equation}
	From \eqref{limits1} we conclude that $v_a^{L}$ is differentiable on $(0,\infty)$, and that the derivative $v_a^{L\prime}(x)$ is given by $\E_x \left[ e^{- q T_a^{+}} ; T_a^{+} < \tau_0^{-} \right]$ for $x>0$. Regarding the continuity of $v_a^{L\prime}$, by an application of the Dominated Convergence Theorem, Lemma \ref{Lem_Converge}, and Lemma 1(ii) in \cite{NobYam2020} we obtain that 
	\[
	v_a^{L\prime}(x+) = \E_x \left[ e^{-q T_{a}^+} ; T_{a}^+ < \tau_0^- \right] = v_a^{L\prime}(x-),\qquad \text{for all $x \in (0,\infty)$}.
	\]

\textit{Step (ii).-} Regarding the capital injections, for $x\in (0,\infty)$ and $\varepsilon>0$ we write $v_a^{R}(x+\varepsilon) - v_a^{R}(x)$ as
\[
v_a^{R}(x+\varepsilon) - v_a^{R}(x) = \E \left[ \int_{{[0, \infty)}} e^{-qt} \diff (R^{(x+\varepsilon)}(t) - R^{(x)}(t)) \right].
\]
Hence, we have
\begin{align}
	v_a^{R}(x+\varepsilon) - v_a^{R}(x) &= \E \left[ \int_{[0,\infty)} q e^{-qt} \left( R^{(x+\varepsilon)}(t) - R^{(x)}(t) \right)  \diff t \right] \\
	& =\E \left[ \int_{[ \tau_0^{(x),-} ,\infty)} q e^{-qt} \left( R^{(x+\varepsilon)}(t) - R^{(x)}(t) \right)  \diff t ; \tau_0^{(x),-} < T_a^{(x),+} \right] \\
	& \geq -\varepsilon \E\left[ e^{-q \tau_0^{(x),-}} ; \tau_0^{(x),-} < T_a^{(x),+} \right]
	=-\varepsilon \E_x\left[ e^{-q \tau_0^{-}} ; \tau_0^{-} < T_a^{+} \right],
	 \label{right3}
\end{align}
where in the first equality we used integration by parts, in the second {equality} we used \eqref{aa02} and the identity concerning $R^{(x+\varepsilon)}-R^{(x)}$ in \eqref{diff_L_2}, and in the first inequality we used Theorem \ref{Prop_diff}(i).\\
On the other hand, we have
\begin{align}
	v_a^{R}(x+\varepsilon) - v_a^{R}(x) &= \E \left[ \int_{[0,\infty)} q e^{-qt} \left( R^{(x+\varepsilon)}(t) - R^{(x)}(t) \right)  \diff t \right] \\
	& \leq  \E\left[ \int_{[\tau_0^{(x+\varepsilon),-},\infty)} q e^{-qt} \left( R^{(x+\varepsilon)}(t) - R^{(x)}(t) \right)  \diff t ; \tau_0^{(x+\varepsilon),-} < T_a^{(x+\varepsilon),+} \right]\\
	& = -\varepsilon \E\left[ e^{-q \tau_0^{(x+\varepsilon),-}} ; \tau_0^{(x+\varepsilon),-} < T_a^{(x+\varepsilon),+} \right]
	=-\varepsilon \E_x\left[ e^{-q \tau_{-\varepsilon}^{-}} ; \tau_{-\varepsilon}^{-} < T_{a-\varepsilon}^{+} \right]
	, \label{right4}
\end{align}
	where in the first equality we used integration by parts and in the second equality we used the identity concerning $R^{(x+\varepsilon)}-R^{(x)}$ in \eqref{diff_R_2}.
Hence, by combining \eqref{right3} and \eqref{right4} we obtain
\begin{equation}\label{right4a}
	-\varepsilon \E_x\left[ e^{-q \tau_0^{-}} ; \tau_0^{-} < T_a^{+} \right] \leq v_a^{R}(x+\varepsilon) - v_a^{R}(x) \leq  -\varepsilon \E_x\left[ e^{-q \tau_{-\varepsilon}^{-}} ; \tau_{-\varepsilon}^{-} < T_{a-\varepsilon}^{+} \right].
\end{equation}
Following similar arguments as in \eqref{right3} and \eqref{right4}, we derive for{ $x> 0$ and $\varepsilon \in (0,x)$:}\\
\begin{equation}\label{left2}
	-\varepsilon  \E_x\left[ e^{-q \tau_{\varepsilon}^{-}} ; \tau_{\varepsilon}^{-} < T_{a+\varepsilon}^{+} \right] \leq v_a^{R}(x) - v_a^{R}(x-\varepsilon) \leq -\varepsilon \E_x\left[ e^{-q \tau_0^{-}} ; \tau_0^{-} < T_a^{+} \right].
\end{equation}
Thus, from \eqref{right4a} and \eqref{left2}; {together with Lemma \ref{Lem_Converge} and Lemma 1(ii) in \cite{NobYam2020}, we obtain, for $x \in (0,\infty)$,}
\[
\lim_{\varepsilon\downarrow 0} \frac{v_a^{R}(x+\varepsilon) - v_a^{R}(x)}{\varepsilon} = \lim_{\varepsilon\downarrow 0} \frac{v_a^{R}(x) - v_a^{R}(x-\varepsilon)}{\varepsilon} = -\E_x \left[ e^{- q \tau_0^{-}} ; \tau_0^{-} < T_a^{+}  \right].
\]
Hence, $v_a^R$ is differentiable on $(0,\infty)$, where its derivative is given by $-\E_x \left[ e^{- q \tau_0^{-}} ; \tau_0^{-} < T_a^{+}  \right]$.
\\
{
Similar computations to those in Step (i) yield the desired continuity of $v_a^{R\prime}$ on $(0,\infty)$. 
}
\end{proof}
Recall that our objective is to determine the existence of a candidate optimal barrier, say $a^\ast$, as the one that satisfies the \textit{smooth fit} condition $v_{a^\ast}'(a^\ast) = 1$. The remainder of this Section is dedicated to proving that, indeed, such a candidate barrier is well-defined.

As a direct consequence of Proposition \ref{Prop401} we obtain that $v_a$ is continuously differentiable on $(0,\infty)$, and its derivative is given by
\begin{equation}\label{density1}
v_a'(x) = \E_x \left[ e^{- q T_a^{+}} ; T_a^{+} < \tau_0^{-} \right] + \beta \E_x \left[ e^{- q \tau_0^{-}} ; \tau_0^{-} < T_a^{+}  \right]. 
\end{equation}
Recall the \lev process with periodic reflection above at $a\geq 0$, denoted by $U_r^a$, that was defined in Section \ref{Sec202}, and the first passage time $\kappa_0^{a,-}(r)$ which was defined in \eqref{fpt_poisson}. Since $U_r^{a}$ is a strong Markov process, we have for $x \geq 0$
\begin{equation}\label{fpt_markov}
	\E_x \left[ e^{-q \kappa_0^{a,-}(r)} \right] = \E_x \left[ e^{-q \tau_0^-} ; \tau_0^{-} < T_a^{+} \right] + \E_x \left[ e^{- q T_a^{+}} ; T_a^{+} < \tau_0^{+} \right] \E_a \left[ e^{-q \kappa_0^{a,-}(r)} \right].
\end{equation}
Hence, using \eqref{fpt_markov} we rewrite \eqref{density1} as
\begin{equation}\label{density2}
	v_a'(x) = \frac{\E_x \left[ e^{-q \kappa_0^{a,-}(r)} \right] + \E_x \left[ e^{-q \tau_0^-} ; \tau_0^{-} < T_a^{+} \right] \left( \beta \E_a \left[ e^{-q \kappa_0^{a,-}(r)} \right] - 1 \right)}{ \E_a \left[ e^{-q \kappa_0^{a,-}(r)} \right] } .
\end{equation}
As a consequence of \eqref{density2}, and since we want $a^\ast$ to satisfy the \textit{smooth fit} condition, we propose our candidate barrier as
\begin{equation}\label{cand_barrier}
	a^{*} := \inf \left\lbrace a\geq 0: \beta \E_a\left[ e^{-q \kappa_0^{a,-}(r)} \right] \leq 1 \right\rbrace.
\end{equation}
{In order to show that $a^\ast$ as proposed in \eqref{cand_barrier} is well-defined and finite, we prove some auxiliary results.}
\begin{lemma}\label{lem_limit_g} Let $g(a): = \beta \E_a\left[ e^{-q \kappa_0^{a,-}(r)} \right]$, then $g$ is non-decreasing with {$\lim_{a \rightarrow \infty}g(a) =0$.}
\end{lemma}
The proof of Lemma \ref{lem_limit_g} {follows from a straightforward application of the dominated convergence theorem} and thus we omit it. 
\begin{lemma}\label{lem_rc}
The function $g$ is continuous on $[0,\infty)$.
\end{lemma}
\begin{proof}
Let $\varepsilon >0$ and $a \geq 0$. Consider the \lev process with periodic reflection above at 0, denoted by $U_r^0$, and we define its first passage time below {the} level $-a$ as
\[
\kappa^{0,-}_{-a}(r) := \inf\lbrace t\geq 0 : U_r^0(t) < -a \rbrace.
\]
	Noting that $\tau_{-a}^-=\kappa^{0,-}_{-a}(r)$ on the event $\{\tau_{-a}^-<T_0^+\}$, and using the strong Markov property we have 
\begin{align}
	g(a) &= \beta {\E\left[ e^{-q \kappa^{0,-}_{-a}(r)} \right]} \notag \\
	&= \beta \E\left[ e^{-q \tau^-_{-a}} ; \tau^-_{-a} < T_0^+ \right] + g(a)\E\left[ e^{-q T_0^+} ; \tau^-_{-a} > T_0^+ \right] \notag 
	\end{align}
and thus
	\begin{align}
	g(a)= \frac{ \beta \E\left[ e^{-q \tau^-_{-a}} ; \tau^-_{-a} < T_0^+ \right] }{ 1 - \E\left[ e^{-q T_0^+} ; \tau^-_{-a} > T_0^+ \right] }. \label{cont1}
\end{align}
It follows that
\[
g(a+\varepsilon) = {\frac{ \beta \E\left[ e^{-q \tau^-_{-a-\varepsilon}} ; \tau^-_{-a-\varepsilon} < T_{0}^+ \right] }{ 1 - \E\left[ e^{-q T_{0}^+} ; \tau^-_{-a-\varepsilon} > T_{0}^+ \right] }}.
\]
Thanks to Assumption \ref{assump1a} and Lemma 1(ii) in \cite{NobYam2020} we have
\begin{align*}
	\lim_{\varepsilon \downarrow 0} {\E\left[ e^{-q \tau^-_{-a-\varepsilon}} ; \tau^-_{-a-\varepsilon} < T_{0}^+ \right]} &= \E\left[ e^{-q \tau^-_{-a}} ; \tau^-_{-a} < T_0^+ \right], \\
	\lim_{\varepsilon \downarrow 0} {\E\left[ e^{-q T_{0}^+} ; \tau^-_{-a-\varepsilon} > T_{0}^+ \right]} &= \E\left[ e^{-q T_0^+} ; \tau^-_{-a} > T_{0}^+ \right],
\end{align*}
hence $\lim_{\varepsilon \downarrow 0} g(a+\varepsilon) = g(a)$, proving {the} right-continuity.\\
The proof of {the} left-continuity follows similarly, by writing $g(a-\varepsilon)$ in terms of expression \eqref{cont1}, we conclude that $\lim_{\varepsilon \downarrow 0} g(a-\varepsilon) = g(a)$ for $a>0$.
\end{proof}
By Lemmas \ref{lem_limit_g} and \ref{lem_rc}, our candidate optimal barrier is well-defined and finite. On the other hand, it may take the value $0$, depending on the path variation of $X$.
{
\begin{remark}\label{remark_optimal}
	\begin{enumerate}
		\item If 0 is regular for $(-\infty,0)$, we have $ \kappa^{0,-}_{0}(r) = 0 $ a.s., hence $g(0) = \beta$ and $a^\ast > 0$. 
		\item If 0 is irregular for $(-\infty,0)$, which occurs when conditions (ii) or (iii) from Theorem 6.5 in \cite{Kyp2014} 
		{do not} hold, then $ \kappa^{0,-}_{0}(r) > 0 $ a.s., hence $ \E\left[ e^{-q \kappa_0^{0,-}(r)} \right] < 1 $. In this case, if $g(0) {\leq1}$, then we set $a^\ast =0$.
	\end{enumerate}
\end{remark}
}
\section{Verification of optimality}\label{sec:verification}
We verify that the periodic-classical reflection strategy at $a^*$ and 0 is optimal, where $a^*$ is given in \eqref{cand_barrier}. We state the main result of this paper, and the remainder of this section is dedicated to its proof.
\begin{theorem}\label{thm_optimal}
	The periodic-classical reflection strategy at $a^*$ and 0 is optimal, thus we have $v(x) = v_{a^*}(x)$ for $x \in [0,\infty)$.
\end{theorem}
First, let $C_{line}^1$ be the set of continuous functions $f:\R\to\R$ of at most linear growth and admits a continuous derivative on $(0,\infty)$, i.e.,
\begin{enumerate}
	\item There exist $a_1,a_2 > 0$ such that $|f(x)| < a_1 + a_2 |x|$ for all $x \in \R$;
	\item $f \in C^1(0,\infty)$.
\end{enumerate}
{In addition, let $C_{line}^2$ denote the subset of $C_{line}^1$ consisting of functions whose derivative $f'$ has a bounded and almost everywhere continuous density on every compact set of $(0,\infty)$.} 
\\
Let $\cL$ be the infinitesimal generator associated with the process $X$ acting on a measurable function $f: \R \rightarrow \R$ belonging to $C_{line}^1(\R)$ (resp., $C_{line}^2(\R)$) 
if $X$ has paths of bounded (resp., unbounded) variation given by
\[
\cL f(x) = \gamma f'(x) + \frac{1}{2} \sigma^2 f''(x) + \int_{\R \backslash \{0\} }\left[ f(x+z) - f(x) - f'(x)z \Ind_{\lbrace |z| < 1 \rbrace} \right] \nu (\diff z).
\]
In the remainder of this section we extend the domain of $v_{\pi}$ to $\R$ by setting $v_{\pi}(x) = \beta x + v_{\pi}(0)$ for $x \in (-\infty,0)$.

We also provide the following verification lemma, which gives us sufficient conditions for optimality. {The proof of the following result is almost the same as that of Lemma 5.3 in \cite{PerYam}, so we omit it.}
\begin{lemma}\label{lem_verification}
	Suppose that $X$ has paths of bounded (resp., unbounded) variation. Let $w: \R \rightarrow \R$ be a function belonging to $C_{line}^1(\R)$ (resp., $C_{line}^2(\R)$) and that satisfies
	\begin{equation}\label{HJB}
		\begin{split}
			(\cL - q) w(x) + r \max_{0 \leq l \leq x} \lbrace l + w(x-l) - w(x) \rbrace &\leq 0, \qquad x > 0,\\ {w'(x)} &\leq \beta, \qquad x > 0, \\
			\inf_{x \geq 0} w(x) &> -m \text{ for some } m > 0.
		\end{split}
	\end{equation}
Then we have $w(x) \geq v(x)$ for all $x \in \R$.
\end{lemma}

The proof of Theorem \ref{thm_optimal} will be divided into several auxiliary lemmas.
\begin{lemma}\label{lemma_smoothness}
	The function $v_{a^*}$ belongs to $C_{line}^1$, it is concave and satisfies
	\[
	0 \leq v'_{a^*}(x) \leq \beta, \quad x \in {(0,\infty)}.
	\]
	In addition, if $X$ has paths of unbounded variation, then $v_{a^*}$ belongs to $C_{line}^2$.
\end{lemma}
\begin{proof}
	First, we consider the case $a^\ast > 0$. Note that thanks to the continuity of the mapping $a \mapsto \beta \E_a\left[ e^{-q \kappa_0^{a,-}(r)} \right] $ and due to expression \eqref{density2}, when $a^\ast > 0$ we can rewrite $v'_{a^\ast}$ as
	\begin{equation}\label{density3}
		v'_{a^\ast}(x) = \beta \E_x\left[ e^{-q \kappa_0^{a^\ast,-}(r)} \right].
	\end{equation}
	We remark that we will use this expression of $v'_{a^\ast}$ throughout the rest of this section.\\
	Due to expression \eqref{density3}, 
	it follows that $0 \leq v'_{a^*}(x) \leq \beta$ and that $x \mapsto v'_{a^*}(x)$ is non-increasing, thus $v_{a^\ast}$ is concave and satisfies the linear growth condition. Since $v_{a^*}$ is continuously differentiable by Proposition \ref{Prop401}, $v_{a^\ast}$ belongs to $C_{line}^1$. In addition, {thanks to Assumption \ref{assump2a} } it follows that $v_{a^\ast}$ belongs to $C_{line}^2$ when $X$ has paths of unbounded variation. \\
	Now, consider the case $a^\ast =0$, which occurs when the conditions from item (2) in Remark \ref{remark_optimal} are satisfied and $g(0)\leq 1$. Then, by expression \eqref{density2} we have, for $x \in (0,\infty)$,
	\begin{align}
		v'_0(x) &= \frac{\E_x \left[ e^{-q \kappa_0^{0,-}(r)} \right] + \E_x \left[ e^{-q \tau_0^-} ; \tau_0^{-} < T_0^{+} \right] \left( \beta \E_0 \left[ e^{-q \kappa_0^{0,-}(r)} \right] - 1 \right)}{ \E_0 \left[ e^{-q \kappa_0^{0,-}(r)} \right] }\\
		&\leq \frac{\E_x \left[ e^{-q \kappa_0^{0,-}(r)} \right]}{\E_0 \left[ e^{-q \kappa_0^{0,-}(r)} \right]} \leq 1,
		\label{inequality}
	\end{align}
	where in the first inequality we used the fact that $\beta \E_0 \left[ e^{-q \kappa_0^{0,-}(r)} \right] \leq 1$, and afterwards we used that $x\mapsto \E_x \left[ e^{-q \kappa_0^{0,-}(r)} \right]$ is non-increasing. On the other hand, using \eqref{density1} we obtain that $v_0'(x) \geq 0$ for $x \in (0,\infty)$. Finally, from Proposition \ref{Prop401} it follows that $v_0 $ belongs to $C_{line}^1$.
\end{proof}
\begin{lemma}\label{lemma_HJB1}
	For $a^* > 0$ we have
	\[
	\max_{0 \leq l \leq x} \lbrace l + v_{a^*}(x-l) - v_{a^*}(x) \rbrace = 
	\begin{cases}
		0, & \text{if } x \in [0,a^*) \\
		x-a^* + v_{a^*}(a^*) - v_{a^*}(x), & \text{if } x \in [a^*,\infty).
	\end{cases}
	\]
\end{lemma}
\begin{proof}
	Let $h_x(l) = l + v_{a^*}(x-l) - v_{a^*}(x)$, then $h_x$ has a derivative given by $h_x'(l) = 1 - v'_{a^*}(x-l)$.
	
		If $x \in [0,a^*)$, then thanks to expression \eqref{density3} we have $h_x'(l) \leq 0$ for all $l \in (0,x)$
		, hence $h_x(l) \leq h_x(0) $ for $l \in [0,x]$, where $h_x(0) =0$.
		
	On the other hand, if $x  \in  [a^*,\infty)$, then due to expression \eqref{density3} we have that $l^* = x - a^*$ is a critical point of $h_x$, and it is a local maximum, thus $h_x(l)\leq h_x(x - a^*) = x-a^* + v_{a^*}(a^*) - v_{a^*}(x)$ for all $l\in[0,x]$.
	
	For $a^*=0$, by Remark \ref{remark_optimal}(2) together with \eqref{inequality}, 
	we have that $h'_x(l)=1 - v'_{a^*}(x-l)\geq0$. Hence, $h_x(l)\leq h_x(x)=x + v_{0}(0) - v_{0}(x)$, for all $l\in[0,x]$.
	
\end{proof}
The following Lemma is the final component for the proof of Theorem \ref{thm_optimal}, we defer its proof to Appendix \ref{App003}.
\begin{lemma}\label{lemma_HJB2}
For $x \in {(0,\infty)}$, we have $ (\cL - q) v_{a^*}(x) + r \max_{0 \leq l \leq x} \lbrace l + v_{a^*}(x-l) - v_{a^*}(x) \rbrace = 0 $.
\end{lemma}

We are now ready to prove Theorem \ref{thm_optimal}
\begin{proof}[Proof of Theorem \ref{thm_optimal}]
By Lemma \ref{lemma_smoothness}, we have that $v_{a^\ast}\in C^1_{line}$ (resp. $v_{a^\ast}\in C^2_{line}$) when $X$ has paths of bounded variation (resp. of unbounded variation). In addition, thanks to Lemma \ref{lemma_HJB2}, {$v_{a^\ast}$} satisfies the first two variational inequalities from \eqref{HJB}. Moreover, as a consequence of Lemma \ref{lemma_smoothness} we also know that $v_{a^\ast}$ is non-decreasing on $(0,\infty)$, which implies that $ \inf_{x \in (0,\infty)}v_{a^\ast}(x) \geq v_{a^\ast}(0) > -\infty $; hence $v_{a^\ast}$ fulfils all the conditions of Lemma \ref{lem_verification}, which gives that $ v_{a^\ast} \geq v $ on $(0,\infty)$. On the other hand, since the periodic-classical reflection strategy $\pi^{0,a^*}$ is admissible we have $  v_{a^*} \leq v$ on $(0,\infty)$. The proof is complete.
\end{proof}

\newpage
	\section*{Acknowledgements}
	Funding in support of this work was provided by a CRM-ISM Postdoctoral Fellowship from the Centre de recherches math\'emathiques (CRM) and the Institut des sciences math\'emathiques (ISM).
Kei Noba was supported by JSPS KAKENHI grant JP21K13807 and JSPS
Open Partnership Joint Research Projects grant JPJSBP120209921. In addition,
Kei Noba stayed at Centro de Investigaci\'on en Matem\'aticas in Mexico as a JSPS Overseas Research
Fellow and received support for the research environment there.

\appendix

\section{The proof of Lemma \ref{lem_smooth}} \label{App002}
	Throughout this proof, we write $H_q(x) = \E_{x}\left[ e^{-q \kappa_0^{a,-}(r)}  \right]$. 
	Assume that $X$ satisfies $\nu(0,\infty) < \infty$. Then we can write the process $X$ as
	\[
	X(t) = X^{\text{SN}}(t) + \sum_{n=1}^{N_\nu(t)} J_n,
	\]
where $X^{\text{SN}}$ is a spectrally negative \lev process with paths of unbounded variation, $N_\nu$ is an homogeneous Poisson process with intensity $\nu(0,\infty)$, and $\{J_n\}_{n \geq 1}$ is a sequence of i.i.d. random variables with distribution $\nu(\cdot \cap (0,\infty))/\nu(0,\infty)$.\\
Let $U_r^{\text{SN},a}$ be the spectrally negative \lev process with periodic reflection above at $a$, and we define $\kappa_0^{\text{SN},a} := \inf\{t>0:  U_r^{\text{SN},a}(t) <0\}$ and
\begin{equation}\label{sn1}
H_\theta^{\text{SN}}(x) := \E_x \left[  e^{-\theta \kappa_0^{\text{SN},a} } \right], \quad \theta >0, x \in (0,\infty).
\end{equation}
We have that $ H_\theta^{\text{SN}} $ is continuously differentiable (see e.g. equation (3.5) in \cite{NobPerYamYan}). 
\\ 
Let $T^{N_\nu}$ be the first jump time of $N_\nu$, which is exponentially distributed with rate $\nu(0,\infty)$. Thanks to the strong Markov property we have 
\begin{align}
	H_q(x) &= \E_x\left[ e^{-q \kappa_0^{\text{SN},a}}; \kappa_0^{\text{SN},a} < T^{N_\nu} \right] + \E_x\left[ e^{-q T^{N_\nu}} H_q(U_r^{\text{SN},a}(T^{N_\nu}) + J_1) ; T^{N_\nu} < \kappa_0^{\text{SN},a} \right] \label{sn3} \\
		&= H^{\text{SN}}_{\hat{q}}(x) + \E_x\left[ e^{-q T^{N_\nu}} H_q(U_r^{\text{SN},a}(T^{N_\nu}) + J_1) ; T^{N_\nu} < \kappa_0^{\text{SN},a} \right],\label{sn2}
\end{align}
where $\hat{q} = q + \nu(0,\infty)$. 
\\
As pointed out in the discussion following \eqref{sn1}, we have that $ H^{\text{SN}}_{\hat{q}} $ has a Radon-Nykodym density w.r.t. the Lebesgue measure, denoted by $ H^{\text{SN} \prime}_{\hat{q}} $, which is bounded on {compact sets of} $(0,\infty)$.\\
It remains to analyse the term $ \E_x\left[ e^{-q T^{N_\nu}} H_q(U_r^{\text{SN},a}(T^{N_\nu}) + J_1) ; T^{N_\nu} < \kappa_0^{\text{SN},a} \right] $. By taking the expectation w.r.t. $T^{N_\nu}$ and $J_1$ we have
\begin{align*}
	\E_x\left[ e^{-q T^{N_\nu}} H_q(U_r^{\text{SN},a}(T^{N_\nu}) + J_1) ; T^{N_\nu} < \kappa_0^{\text{SN},a} \right] & = \E_x \left[ \int_0^{\kappa_0^{\text{SN},a}} e^{- \hat{q} t} \left( \int_{{(0, \infty)}} H_q(U_r^{\text{SN},a}(t) + u) \nu(\diff u) \right) \diff t \right].
\end{align*}
To simplify the notation, we write $\mathcal{H}(y):= \int_{{(0, \infty)}} H_q(y + u) \nu(\diff u)$. Then, it is a direct consequence of Proposition 5.3 in \cite{MMNP} that we can write
\begin{align*}
	\E_x \left[ \int_0^{\kappa_0^{\text{SN},a}} e^{- \hat{q} t} \mathcal{H}(U_r^{\text{SN},a}(t)) \diff t \right] & = K_1 W_a^{(\hat{q},r)}(x) - \rho_a^{(\hat{q},r)}(x; \mathcal{H}) - K_2 \bar{W}^{(\hat{q}+r)}(x-a)\\
	& \quad  - \int_0^{x-a} \mathcal{H}(y+a) W^{(\hat{q}+r)}(x-a-y) \diff y,
\end{align*}
where $K_1, K_2$ are constants, and $W_a^{(\hat{q},r)}(x),\, \bar{W}^{(\hat{q}+r)}, \, \rho_a^{(\hat{q},r)}(x; \mathcal{H})$ are defined in terms of the scale functions $W^{\hat{q}}, \, W^{(\hat{q}+r)}$ (see equations (2.8) and (2.11) in \cite{MMNP}). Since $X^{SN}$ has paths of unbounded variation, the scale function is $C^1$ on $(0,\infty)$, so it follows that $x \mapsto \E_x \left[ \int_0^{\kappa_0^{\text{SN},a}} e^{- \hat{q} t} \mathcal{H}(U_r^{\text{SN},a}(t)) \diff t \right]$ is $C^1$ on $(0,\infty)$ as well. As a consequence, its first derivative is a version of the Radon-Nykodym density on w.r.t. to the Lebesgue measure, and since it is continuous, it is also bounded on {compact sets of} $(0,\infty)$.

It now follows easily that $H_q \in C^1(0,\infty)$, thus it has a Radon-Nykodym density w.r.t. to the Lebesgue measure, and we can take a version of it that coincides with $H_q'$, which is bounded on {compact sets of} $(0,\infty)$.\\
{The case where $\nu(-\infty,0) < \infty$ is handled similarly, using the corresponding fluctuation identities found in Corollary 3.1(ii) of \cite{AvrPerYam} and Lemma 3.4 of \cite{BoWangYan}.} 

\section{Proof of Lemma \ref{lemma_HJB2}}\label{App003}
\begin{proof}
	First, due to Lemma \ref{lemma_HJB1} we have $ \max_{0 \leq l \leq x} \lbrace l + v_{a^*}(x-l) - v_{a^*}(x) \rbrace = (x-a^*)_+ + v_{a^*}(x\wedge a^*) - v_{a^*}(x) $ for $x \in [0,\infty)$, so it suffices to show that $ (\cL - (r+q)) v_{a^*}(x) + r h(x)  = 0 $, where $h(x):= (x-a^*)_+ + v_{a^*}(x\wedge a^*)$.\\
	Let $T_1$ 
	be the first arrival time of the Poisson process {$N_r$}, then since $\lbrace U_r^{0,a^*}(t) : t < T_1\wedge \tau_0^{-} \rbrace$ is equal in law to $\lbrace X(t) : t < T_1\wedge \tau_0^{-} \rbrace$, and due to the Strong Markov property we have
	\begin{align}
		v_{a^*}(x) &= \E_x\left[ e^{-q T_1} h(X(T_1)) ; T_1 < \tau_0^-\right] + \beta \E_x\left[ e^{-q \tau_0^{-}} X(\tau_0^{-}) ; \tau_0^{-} < T_1 \right] \notag\\&+ v_{a^*}(0) \E_x\left[ e^{-q \tau_0^{-}} ; \tau_0^{-} < T_1 \right] \notag \\
		&= r \E_x \left[ \int_0^{\tau_0^{-}} e^{-(r+q)s} h(X(s)) \diff s \right] + \beta \E_x\left[ e^{-(r+q) \tau_0^{-}} X(\tau_0^{-}) ; \tau_0^{-} < \infty \right] \notag\\&+ v_{a^*}(0) \E_x\left[ e^{-(r+q) \tau_0^{-}} ; \tau_0^{-} < \infty \right].\label{aux1}
	\end{align}
	Now, the strong Markov property gives, for $t \geq 0$:
	\begin{align}
		\E_x \Bigg[ \int_{0}^{\tau_0^{-}}& e^{-(r+q)s} h(X(s)) \diff s  \Bigg| \cF_t \Bigg] \notag\\&= \Ind_{\lbrace \tau_0^{-} \leq t \rbrace} \int_{0}^{\tau_0^{-}} e^{-(r+q)s} h(X(s)) \diff s + \E_x \left[ \Ind_{\lbrace \tau_0^{-} > t \rbrace} \int_{0}^{\tau_0^{-}} e^{-(r+q)s} h(X(s)) \diff s  \Bigg| \cF_t \right] \notag \\
		&=  \int_{0}^{t \wedge \tau_0^{-}} e^{-(r+q)s} h(X(s)) \diff s + e^{-(r+q)t} \Ind_{\lbrace \tau_0^{-} > t \rbrace} \E_{X(t)} \left[ \int_{0}^{\tau_0^{-}} e^{-(r+q)s} h(X(s)) \diff s \right] \notag \\
		&= \int_{0}^{t \wedge \tau_0^{-}} e^{-(r+q)s} h(X(s)) \diff s + e^{-(r+q)(t\wedge \tau_0^{-})}  \E_{X(t \wedge \tau_0^{-})} \left[ \int_{0}^{\tau_0^{-}} e^{-(r+q)s} h(X(s)) \diff s \right]. \label{aux2}
	\end{align}
	The previous identity implies that the right-hand side of \eqref{aux2} is a martingale w.r.t. $\p_x$. \par By a similar computation we obtain that the processes $\left\{e^{-(r+q)t\wedge \tau_0^{-}} \E_{X(t \wedge \tau_0^{-})} \left[ e^{-(r+q)\tau_0^{-}} ; \tau_0^{-} < \infty \right]\right\}_{t\geq0}$ and $\left\{e^{-(r+q)t\wedge \tau_0^{-}} \E_{X(t \wedge \tau_0^{-})} \left[ e^{-(r+q)\tau_0^{-}} X(\tau_0^{-}) ; \tau_0^{-} < \infty \right]\right\}_{t\geq0}$, are also martingales w.r.t. to $\p_x$. It follows that
	\begin{equation}
		r \int_{0}^{t \wedge \tau_0^{-}} e^{-(r+q)s} h(X(s)) \diff s + e^{-(r+q)t\wedge \tau_0^{-}} v_{a^*}(X(t\wedge \tau_0^{-})),\qquad t\geq0,
	\end{equation}
	is a martingale w.r.t. to $\p_x$.\\
	Due to Lemma \ref{lemma_smoothness}, $v_{a^*}$ is smooth enough, hence by an application of the It\^{o}--Meyer formula, we have
	\begin{align}
		r \int_{0}^{t \wedge \tau_0^{-}} e^{-(r+q)s} h(X(s)) \diff s + & e^{-(r+q)t\wedge \tau_0^{-}}  v_{a^*}(X(t\wedge \tau_0^{-})) - v_{a^*}(x) \\ 
		& = \int_{0}^{t \wedge \tau_0^{-}} e^{-(r+q)s} \left( (\cL - q - r) v_{a^{*}}(X(s)) + r h(X(s)) \right) \diff s + M(t\wedge \tau_0^{-}), \label{aux3}
	\end{align}
	where $M(t\wedge \tau_0^{-})$ is a local martingale w.r.t. $\p_x$.\\
		Since the left-hand side of \eqref{aux3} and $M(t\wedge \tau_0^{-})$ are local martingales, and because $v_{a^\ast} \in C^1$ due to Lemma\ref{lemma_smoothness}, then following a similar argument to that in the proof \cite[(12)]{BifKyp2010} and \cite[Lemma 5.7]{Nob2019}, we have
	\begin{equation}\label{aux4}
	(\cL - q - r) v_{a^{*}}(x) + r h(x) =0, \quad x \in (0,\infty),
	\end{equation}
	when $X$ has paths of bounded variation. In addition, when $X$ has paths of unbounded variation, we use Assumption \ref{assump2a} and the argument of \cite[Lemma 5.7]{Nob2019} in order to take a version of the density of $v_{a^\ast}'$, denoted by $v_{a^\ast}''$, so that it is continuous and satisfies \eqref{aux4} on $(0,\infty)$. The proof is complete.
\end{proof}

\end{document}